\pdfoutput=1
\documentclass[microtype]{gtpart}
\usepackage{graphicx}
\usepackage[mathscr]{eucal}
\usepackage{amssymb}
\usepackage[parfill]{parskip}
\usepackage[dvipsnames]{xcolor}
\usepackage{enumerate}
\usepackage{hyperref}
\usepackage[all]{xy}

\newcommand{\br}{\mathbb{R}}
\newcommand{\bc}{\mathbb C}
\newcommand{\bz}{\mathbb Z}
\newcommand{\bn}{\mathbb N}

\newcommand{\si}{\sigma}
\newcommand{\vp}{\varphi}

\newcommand{\sd}{\mathscr D}

\newcommand{\ssm}{\smallsetminus}
\newcommand{\into}{\hookrightarrow}

\DeclareMathOperator{\homeo}{Homeo}
\DeclareMathOperator{\Ends}{Ends} 
\DeclareMathOperator{\genus}{genus}
\DeclareMathOperator{\Aut}{Aut}
\DeclareMathOperator{\mcg}{Map}
\DeclareMathOperator{\pmcg}{PMap}
\DeclareMathOperator{\out}{Out}

\newtheorem{Thm}{Theorem}[section]
\newtheorem{Thm*}{Theorem}
\newtheorem{Prop}[Thm]{Proposition}
\newtheorem{Lem}[Thm]{Lemma}
\newtheorem{Cor}[Thm]{Corollary}
\newtheorem{Cor*}[Thm*]{Corollary}
\newtheorem{Conj}[Thm]{Conjecture}
\newtheorem{Question}[Thm]{Question}
\newtheorem{Problem}[Thm]{Problem}
\newtheorem*{Remark*}{Remark}

\theoremstyle{definition}
\newtheorem{Def}[Thm]{Definition}

\title[Algebraic and topological properties of big mapping class groups]{Algebraic and topological properties \\ of big mapping class groups}

\author{Priyam Patel}
\address{Department of Mathematics, University of California \\ Santa Barbara, CA 93106}
\email{patel@math.ucsb.edu}

\author{Nicholas G. Vlamis}
\address{Department of Mathematics, University of Michigan \\ Ann Arbor, MI 48109}
\email{vlamis@umich.edu}

\begin{document}  

\begin{abstract}
Let $S$ be an orientable, connected surface with infinitely-generated fundamental group.
The main theorem states that if the genus of \( S \) is finite and at least 4, then the isomorphism type of the pure mapping class group associated to $S$, denoted $\pmcg(S)$, detects the homeomorphism type of $S$.  
As a corollary, every automorphism of $\pmcg(S)$ is induced by a homeomorphism, which extends a theorem of Ivanov from the finite-type setting.
In the process of proving these results, we show that $\pmcg(S)$ is residually finite if and only if $S$ has finite genus, demonstrating that the algebraic structure of $\pmcg(S)$ can distinguish finite- and infinite-genus surfaces. 
As an independent result, we also show that $\mcg(S)$ fails to be residually finite for any infinite-type surface $S$. 
In addition, we give a topological generating set for $\pmcg(S)$ equipped with the compact-open topology.
In particular, if $S$ has at most one end accumulated by genus, then $\pmcg(S)$ is topologically generated by Dehn twists, otherwise the Dehn twists along with handle shifts topologically generate.
\end{abstract}

\maketitle


\section{Introduction}

A surface is of finite type if its fundamental group is finitely generated; otherwise, it is of infinite type.
Throughout, all surfaces are assumed to be connected, orientable, and to have compact (possibly empty) boundary.

The \emph{mapping class group}, denoted $\mcg(S)$, of a surface $S$ is the group of orientation preserving homeomorphisms of $S$ up to isotopy, where we require all homeomorphisms to fix $\partial S$ pointwise. 
The algebraic and geometric structure of mapping class groups of finite-type surfaces is generally well understood. In contrast, very little is known about \emph{big} mapping class groups, i.e. mapping class groups of infinite-type surfaces.  

Big mapping class groups arise naturally from the study of group actions on surfaces (e.g. \cite{CalegariCircular}), taut foliations of 3-manifolds (e.g. \cite{CantwellEndperiodic}), and the Artinization of groups (e.g. \cite{FunarUniversal, FunarBraided}). 
(See \cite{CalegariBig2} for a detailed discussion of these connections.)
There has been a recent trend aimed at  understanding infinite-type surfaces and their mapping class groups.
For instance, recent work has investigated big mapping class group actions on hyperbolic graphs (e.g. \cite{BavardHyperbolic, AramayonaArc, DurhamGraphs}).

This article focuses on the algebraic structure of big mapping class groups.
Our methods use the language of topological groups and initiates the study of big mapping class groups in this category.
The motivation for our work is a type of algebraic rigidity question for mapping class groups:

\begin{Question} \label{ques:mainquestion}
Does the isomorphism type of $\mcg(S)$ determine the topology of $S$?
In particular, if $\mcg(S)$ is isomorphic to $\mcg(S')$, then are $S$ and $S'$ homeomorphic?
\end{Question}

\begin{Remark*}
Since the original submission of this article, Question \ref{ques:mainquestion} has been answered in the affirmative by Bavard--Dowdall--Rafi \cite{BavardIsomorphism}.
Their techniques are different than the those used in this article in the proof of Theorem \ref{thm:main}.
\end{Remark*}

Restricting to the finite-type setting -- with the exception of low-complexity cases -- this question has a positive answer.
To see this observe that the rank of the center of $\mcg(S)$, the virtual cohomological dimension of $\mcg(S)$ \cite{HarerVirtual}, and the rank of a maximal abelian subgroup in $\mcg(S)$ \cite{BirmanAbelian} together determine the genus, number of punctures, and the number of boundary components of $S$ and thus its topological type.

Observe that $\mcg(S)$ is countable if and only if $S$ is of finite type. Therefore, the isomorphism type of $\mcg(S)$ can distinguish between finite- and infinite-type surfaces.  
This reduces the question to considering specifically infinite-type surfaces, where little is known.
The invariants mentioned for the finite-type setting, with the exception of the rank of the center, are all infinite for big mapping class groups.  

In progress towards this question, we change our focus to a natural subgroup of $\mcg(S)$.
In particular, we will work with the \emph{pure mapping class group}, denoted $\pmcg(S)$, consisting of mapping classes acting trivially on the topological ends of $S$.
In this setting we address the same question and provide a partial answer:

\begin{Thm*}
\label{thm:main}
Let $S$ be a surface whose genus is finite and at least 4.
If $S'$ is a surface, then any isomorphism between $\pmcg(S)$ and $\pmcg(S')$ is induced by a homeomorphism.
\end{Thm*}

If $S$ has empty boundary, then the \emph{extended} mapping class group, denoted $\mcg^\pm(S)$, is the degree-2 extension of $\mcg(S)$ that includes orientation-reversing mapping classes; otherwise, we set $\mcg^\pm(S) = \mcg(S)$.
Setting $S' = S$ in Theorem~\ref{thm:main} yields a generalization of a celebrated result of Ivanov for finite-type mapping class groups \cite{IvanovAutomorphismsT}: 

\begin{Cor*}
\label{cor:automorphism}
If $S$ is a borderless surface whose genus is finite and at least 4, then the natural monomorphism from $\mcg^\pm(S)$  to $\Aut(\pmcg(S))$ is an isomorphism.
\end{Cor*}

An essential aspect of the proof of Theorem \ref{thm:main}, just as in Ivanov's work, is to give a characterization of Dehn twists (Proposition \ref{prop:characterization}), which allows us to conclude that an isomorphism of pure mapping class groups must send Dehn twists to Dehn twists. To obtain this characterization, we introduce a non-standard topology $\tau_w$ on $\pmcg(S)$, described in detail in Section \ref{sec:initial}.
In Corollary \ref{cor:ends}, we discuss the structure of $\out(\pmcg(S))$; in particular, it is isomorphic to the automorphism group of a countable boolean algebra.

Due to Theorem \ref{thm:main}, it is natural for one to expect that $\pmcg(S)$ is a characteristic subgroup of $\mcg(S)$.
If this is the case, then Corollary \ref{cor:automorphism} can be extended to $\mcg(S)$ to again match the finite-type setting.
We state these as conjectures with the same assumptions as in Theorem \ref{thm:main}; however, we expect these conjectures and Corollary \ref{cor:automorphism} to hold in general.

\begin{Conj}
Let $S$ be a surface whose genus is finite and at least 4.
\begin{enumerate}
\item
$\pmcg(S)$ is a characteristic subgroup of $\mcg(S)$ and $\mcg^\pm(S)$.
\item
If \( S \) is borderless, then 
$\Aut(\mcg^\pm(S)) = \Aut(\mcg(S)) = \mcg^\pm(S)$
\end{enumerate}
\end{Conj}

(This conjecture has been subsequently proven -- in greater generality -- by Bavard--Dowdall--Rafi \cite{BavardIsomorphism}.)

The first step in proving Theorem \ref{thm:main} is to distinguish finite- and infinite-genus surfaces via the algebraic structure of their associated pure mapping class groups.
In order to do this, we investigate whether big mapping class groups are residually finite.
As these results are of independent interest, we include Theorem \ref{thm:rf} below as a summary of the results in Section \ref{sec:residual}.
Recall that a group is \emph{residually finite} if the intersection of its proper finite-index normal subgroups is trivial.

\begin{Thm*}
\label{thm:rf}
Let $S$ be any surface.
\begin{enumerate}
\item
$\pmcg(S)$ is residually finite if and only if $S$ has finite genus.  
\item
$\mcg(S)$ is residually finite if and only if $S$ is of finite type.
\end{enumerate}
\end{Thm*}

Note that the finite-type cases are handled by Grossman \cite{GrossmanResidual}. Using the structure of $\tau_w$, we show that $\pmcg(S)$ is residually finite whenever $S$ has finite genus (Proposition~\ref{prop:fgrf}). 
We then introduce the infinite-stranded braid group $B_\infty$, show that it fails to be residually finite (Corollary \ref{cor:braid}), and prove that it embeds in every finite-genus big mapping class group (Proposition \ref{prop:braid}). Thus, $\mcg(S)$ is not residually finite when $S$ is of infinite type and has finite genus.
Proposition \ref{prop:no} shows that $\pmcg_c(S)$, the subgroup of $\mcg(S)$ consisting of mapping classes with compact support, has no finite quotients when $S$ has infinite genus. This fact is almost immediate from a theorem of Paris \cite{ParisSmall} and implies that when $S$ has infinite genus, $\mcg(S)$ and $\pmcg(S)$ are not residually finite.

The reason $\pmcg(S)$ of a finite-genus surface behaves differently than other big mapping class groups is the existence of forgetful homomorphisms to finite-type pure mapping class groups.
We use these forgetful homomorphisms to build the topology $\tau_w$ on $\pmcg(S)$  mentioned above (see Section~\ref{sec:initial}).
Understanding basic topological properties of $\tau_w$ is the key to characterizing Dehn twists and understanding the residual properties of $\pmcg(S)$.
The construction of $\tau_w$ naturally leads to an inverse limit of finite-type pure mapping class groups, which is the completion (as uniform spaces) of the associated big pure mapping class group.
This inverse limit viewpoint is thoroughly discussed in Section \ref{sec:inverse}.

As Dehn twists and the topology of $\tau_w$ play a critical role in our understanding of $\pmcg(S)$, it is natural to understand which of these properties hold in the more standard compact-open topology.
Equipping the group of homeomorphisms of a surface with the compact-open topology, we  give $\mcg(S)$ the corresponding quotient topology.
We let $\tau_q$ denote the restriction of this topology to $\pmcg(S)$.  
Note that if $S$ is of finite type, then $\tau_q$ is the discrete topology.
From the definition of $\tau_w$ and the Dehn-Lickorish Theorem, it follows that Dehn twists topologically generated $(\pmcg(S),\tau_w)$.

In the following theorem, we see that Dehn twists topologically generated \( (\pmcg(S), \tau_q) \) if \( S \) is finite genus; however, this can fail in the infinite-genus setting.
Intuitively, a topological end of a surface is \emph{accumulated by genus} if there is a sequence of handles converging to the end (see Section \ref{sec:ends} for the precise definition).
In order to build a topological generating set in the infinite-genus setting we introduce the notion of a \emph{handle shift} in Section \ref{sec:cpt}.
Roughly speaking, a handle shift is a homeomorphism which applies a \( +1 \) shift to a \( \bz \)-indexed collection of handles in the surface.
This notion allows us to build a topological generating set for \( (\pmcg(S), \tau_q) \):

\begin{Thm*}\label{thm:dense}
The set of Dehn twists topologically generate $(\pmcg(S), \tau_q)$ if and only if $S$ has at most one end accumulated by genus.
If $S$ has at least two ends accumulated by genus, then the set of Dehn twists together with the set of handle shifts topologically generate $(\pmcg(S), \tau_q)$. 
\end{Thm*}

\subsection{Outline}
In Section \ref{sec:background} we give the necessary background, which focuses on the structure of the space of ends of an infinite-type surface.

In Section \ref{sec:initial}, we define the topology $\tau_w$, show it is Hausdorff (Lemma \ref{lem:Hausdorff}), and investigate the closure of the set of Dehn twists (Lemma \ref{lem:closure}), allowing us to characterize Dehn twists (Proposition \ref{prop:characterization}).
Section \ref{sec:residual} is dedicated to investigating residual finiteness.

Given the results of Section \ref{sec:residual}, Section \ref{sec:automorphisms} focuses on finite-genus surfaces.  In this section, we combine the characterization of Dehn twists from Section \ref{sec:initial} and results regarding homomorphisms between finite-type mapping class groups to conclude the proof of Theorem \ref{thm:main} in Propositions \ref{thm:genus} and \ref{prop:homeo}.

In Section \ref{sec:cpt}, we introduce the notion of a handle shift and prove Theorem \ref{thm:dense}.
In Section \ref{sec:inverse}, we discuss inverse limits of finite-type pure mapping class groups and finish the paper by showing the proper containment $\tau_w \subset \tau_q$  in Subsection \ref{sec:comparing}.

\section*{Acknowledgements}

The authors would like to thank Javier Aramayona, Ara Basmajian, Daryl Cooper, David Fern\'andez-Bret\'on, Chris Leininger, Darren Long, and Scott Schneider for several helpful conversations.

The second author was supported in part by NSF RTG grant 1045119.


\section{Background}\label{sec:background}

\subsection{Curves}
A \emph{simple closed curve} in a surface $S$ is the image of a topological embedding $\mathbb{S}^1 \hookrightarrow S$.
A simple closed curve is \emph{trivial} if it is homotopic to a point; it is \emph{peripheral} if it is either homotopic to a boundary component or bounds a once-punctured disk in \( S \); it is \emph{essential} if it is neither trivial nor peripheral; it is \emph{separating} if its complement is disconnected and \emph{nonseparating} otherwise.
If $c$ is a simple closed curve in $S$, then we let $[c]$ denote its homotopy class in $S$.
Given two homotopy classes of simple closed curves $[c]$ and $[d]$, their \emph{geometric intersection number}, denoted $i([c],[d])$, is defined to be 
\[
i([c],[d]) = \min \{|c' \cap d'| \co [c'] = [c] \text{ and } [d']=[d]\}.
\]
When is it clear from context, we will conflate a simple close curve with its isotopy class.

Recent work of Hern\'andez--Morales--Valdez \cite{Hernandez} assures us that, just as in the finite-type setting, simple closed curves play an essential role in the study of big mapping class groups.
In particular, they extend the Alexander method to infinite-type surfaces \cite[Theorem 1.1]{Hernandez}, which yields the following lemma:

\begin{Lem}[{\cite[Corollary 1.2]{Hernandez}}]
\label{lem:hernandez}
Suppose $S$ is an infinite-type surface without boundary.
If an element of $\mcg^\pm(S)$ fixes the isotopy class of every simple closed curve, then it is trivial.
\end{Lem}

We will need one final notion: If $F \subset S$ is a subsurface, then we call $F$ \emph{essential} if it is connected and the inclusion $F\hookrightarrow S$ induces a monomorphism $\pmcg(F) \hookrightarrow \pmcg(S)$, where a homeomorphism of $F$ is extended to $S$ by the identity.

\subsection{Ends}
\label{sec:ends}
In order to work with infinite-type surfaces, it is essential to understand their space of ends, which we now define.

\begin{Def}
Let $S$ be an infinite-type surface and fix an exhaustion $\{K_i\}_{i\in \bn}$ of $S$ by compact subsurfaces, that is, each $K_i$ is a compact surface of $S$, $K_i \subset K_{i+1}$, and $\bigcup_{i\in \bn} K_ i = S$. 

\begin{itemize}
\item An \emph{end of $S$} is a sequence $\{U_i\}_{i\in \bn}$ such that each $U_i$ is a connected component of $S \ssm K_i$ and $U_{i+1} \subset U_i$ for all $i\in \bn$. 

\item Define $U^*_i$ to be the set of all ends that contain $U_i$, where $U_i$ is a connected component of $S \ssm K_i$. The \emph{space of ends} of $S$, denoted $\Ends(S)$, consists -- as a set -- of the ends of $S$. It is equipped with the topology generated by sets of the form $U^*_i$.
\end{itemize}
\end{Def}

The space $\Ends(S)$ is totally disconnected, separable, and compact (see \cite[Chapter 1, \S36 \& \S37]{AhlforsRiemann}); hence, it can be realized as a closed subset of the Cantor set.
An end of $S$ comes in two types: 

\begin{itemize}
\item An end $\{U_i\}$ is \emph{accumulated by genus} if $U_i$ has infinite genus for all $i$. 
\item An end $\{U_i\}$ is \emph{planar} if $U_i$ is planar for all but finitely many $i$.
\end{itemize}

In addition, an end is \emph{isolated} if it is an isolated point in the topology given above.
We will use the word \emph{puncture} to denote an isolated planar end.

Ker\'{e}kj\'{a}rt\'{o} (see \cite{RichardsClassification}) showed that the homeomorphism type of an orientable infinite-type surface is determined by the quadruple 
\[
(g, b,  \Ends_\infty(S), \Ends(S)),
\]
where $g\in \mathbb{N}\cup \{0,\infty\}$ is the genus of $S$, $b \in \bn$ is the number of boundary components, and $\Ends_\infty(S)$ is the subset of $\Ends(S)$ accumulated by genus. 

An immediate consequence of the classification of infinite-type surfaces is:

\begin{Lem}[Definition of $\overline S$]
When $S$ has finite genus, there exists a unique -- up to homeomorphism -- compact surface, which we label $\overline S$, containing $S$ such that $\overline S \ssm S$ is homeomorphic to $\Ends(S)$.
\end{Lem}

Given the definition of $\Ends(S)$ it is clear that a homeomorphism of $S$ induces a homeomorphism of $\Ends(S)$.
As any isotopy must fix the ends of $S$, we see that there is an induced action of $\mcg(S)$ on $\Ends(S)$.
The \emph{pure mapping class group} $\pmcg(S)$ is the kernel of this action.

\begin{Prop}[{\cite[Proposition I.1]{BeguinConstruction}}]
\label{prop:extension}
Let $\mathbb{D}$ denote the unit disk and let $\Pi$ be a totally-disconnected compact subset contained in the interior of $\mathbb{D}$.
Given a homeomorphism $f\co \Pi \to \Pi$ there exists a homeomorphism $\hat f \co \mathbb{D} \to \mathbb{D}$  and a regular neighborhood $N$ of $\partial \mathbb{D}$ such that $\hat f |_\Pi = f$ and $\hat f|_{N} = id$. 
\end{Prop}

Observe that since a regular neighborhood of the boundary of $\mathbb{D}$ is fixed pointwise, it follows that $\hat f$ is orientation preserving.
This yields:

\begin{Cor}
\label{cor:extension}
If $S$ has finite genus, then the induced homomorphism $\mcg(S) \to \homeo(\Ends(S))$ is surjective.
\end{Cor}


\section{The initial topology}
\label{sec:initial}

For the entirety of this section, we require $S$ to have finite genus and be of infinite type.
In this setting, we will introduce a topology on $\pmcg(S)$ that is suited to our algebraic study of the group. 
Define the set $\Lambda(S)$ to be the set of finite subsets of $\Ends(S)$, that is, 
\[
\Lambda(S) = \{\lambda \subset \Ends(S) : |\lambda| < \infty\}.
\]
When it is clear from context, we will simply write $\Lambda$.
For each $\lambda \in \Lambda$, we define the surface $S_\lambda = \overline S \ssm \lambda$.
The inclusion $i_\lambda\co S \into S_\lambda$ induces a homomorphism $\vp_\lambda \co \pmcg(S) \to \pmcg(S_\lambda)$.
This homomorphism is the forgetful homomorphism, where one forgets all ends of $S$ not in $\lambda$. 

The \emph{initial topology}, denoted $\tau_w$, on $\pmcg(S)$ is the initial topology with respect to the family of maps $\{\vp_\lambda\}_{\lambda \in \Lambda}$, that is, it is the coarsest topology such that $\vp_\lambda$ is continuous for each $\lambda \in \Lambda$.
(We are viewing $\pmcg(S_\lambda)$ with the discrete topology.)
If at this point the reader is inclined to think about inverse limits, we discuss this viewpoint in Section~\ref{sec:inverse}.

For the rest of the section our goal is to present two key lemmas (Lemmas \ref{lem:Hausdorff} and \ref{lem:closure}) regarding the algebraic structure of $\pmcg(S)$.
Before starting, we need to understand the center of \( \pmcg(S) \), which we denote by \( Z(\pmcg(S)) \).
Recall that the pure mapping class group of any finite-type surface is generated by Dehn twists (see \cite[Section 4.4.4]{Primer}).
Combining this fact with the definition of the initial topology above, we have the following lemma as an immediate corollary:

\begin{Lem}
\label{lem:generate1}
\( (\pmcg(S), \tau_w) \) is topologically generated by Dehn twists.
\end{Lem}

\begin{Lem}
\label{lem:center2}
An element of \( \pmcg(S) \) is central if and only if it fixes the isotopy class of every simple closed curve.
\end{Lem}

\begin{proof}
Let \( f \in \pmcg(S) \).
Recall that if \( c \) is a simple closed curve in \( S \), then \( f \cdot T_{c} \cdot f^{-1} = T_{f(c)}\). 
Now assume that \( f \) fixes the isotopy class of every simple closed curve.
It follows that \( f \cdot T_{c} \cdot f^{-1} = T_{f(c)} = T_c \) for any simple closed curve \( c \).
Therefore, \( f \) commutes with every Dehn twist in \( \pmcg(S) \) and hence, by Lemma \ref{lem:generate1}, \( f \) commutes with every element of a topological generating set for \( \pmcg(S) \).
It is an easy exercise to show that  multiplication in $(\pmcg(S), \tau_w)$ is continuous.
Therefore, we can conclude that \( f \in Z(\pmcg(S)) \).

Now assume that \( f \) is central.
It follows that \( T_c = f \cdot T_{c} \cdot f^{-1} = T_{f(c)} \) and hence \( f([c]) = [c] \) for every simple closed curve \( c \) in \( S \).
\end{proof}

\begin{Lem}
\label{lem:center}
\( Z(\pmcg(S)) \) is generated by the Dehn twists about the boundary components of \( S \).
In particular, if \( S \) is borderless then \( \pmcg(S) \) has trivial center.
\end{Lem}

\begin{proof}
Let \( G < \pmcg(S) \) be the subgroup generated by the Dehn twists about the boundary components of \( S \).

Let \( R \) be the surface obtained from \( S \) by gluing a once-punctured disk to each component of \( \partial S \).
Let \( \iota\co S \hookrightarrow R \) be the associated inclusion and \( \iota_* \co \pmcg(S) \to \pmcg(R) \) the induced homomorphism.
By construction, \( \iota_* \) is surjective.
To see this, note that \( \iota_* \) is continuous and every Dehn twist in \( \pmcg(R) \) is an image of a Dehn twist in \( \pmcg(S) \).

As \( \iota_* \) is surjective, it follows that if \( f \in Z(\pmcg(S)) \), then \( \iota_*(f) \in Z(\pmcg(R)) \).
By Lemma \ref{lem:center2}, \( \iota_*(f) \) fixes the isotopy class of every simple closed curve; therefore, as \( R \) is borderless, we can apply Lemma \ref{lem:hernandez} to see that\( \iota_*(f) \) is trivial.

A direct replacement of the finite-type Alexander method with the infinite-type Alexander method \cite{Hernandez} in the proof of \cite[Theorem 3.18]{Primer} implies \( G = \ker \iota_* \); in particular,  \( f \in G \).
Therefore, \( Z(\pmcg(S)) \subset G \).

Now, let \( f \in G \), then \( f \) fixes the isotopy class of every simple closed curve.
Lemma \ref{lem:center2} allows us to conclude that \( f \) is central and  \( G \subset Z(\pmcg(S)) \).
\end{proof}

\begin{Lem}\label{lem:Hausdorff}
$\tau_w$ is Hausdorff.
\end{Lem}

\begin{proof}
To show the Hausdorff property, it is enough to show that any point can be separated from the identity with open sets.
By the definition of $\tau_w$ and continuity of left multiplication, it is enough to show that
\[
\bigcap_{\lambda \in \Lambda} \ker \vp_\lambda = \{id\}.
\]
Let $f \in \pmcg(S)$ be a nontrivial element.
If $f$ is central in $\pmcg(S)$, then by Lemma \ref{lem:center} it is a product of Dehn twists about boundary components and hence the same is true for $\vp_\lambda(f)$.
In particular, $f$ is not an element of $\ker \vp_\lambda$ whenever $|\lambda| > 1$.
Assuming $f$ is not central,  by Lemma \ref{lem:center2}, there exists an essential simple closed curve $a$ in $S$ such that $f(a)$ is not homotopic to $a$.
Choose an essential finite-type surface $F \subset S$ so that $\genus(F) = \genus(S)$ and both $a$ and $f(a)$ are contained in $F$.
Let $\lambda \in \Lambda$ be such that the intersection of $\lambda$ with each component of $\overline S \ssm F$ is nonempty.
It follows that $i_\lambda |_F$ is an embedding of $F$ into $S_\lambda$ as an incompressible surface; in particular, $i_\lambda(f(a))$ and $i_\lambda(a)$ are not homotopic in $S_\lambda$.
It follows that $\vp_\lambda(f)$ does not fix the homotopy class of $i_\lambda(a)$ in $S_\lambda$; hence, $f \notin \ker\vp_\lambda$.
\end{proof}

The proof of Lemma \ref{lem:Hausdorff} tells us that $\pmcg(S)$ inherits the residual properties of finite-type pure mapping class groups.
In particular,  Lemma \ref{lem:Hausdorff} combined with the work of Grossman \cite{GrossmanResidual} implies that $\pmcg(S)$ is residually finite.
This is discussed in more depth in  Section \ref{sec:residual}.
In addition, it follows from the inverse limit construction in Section \ref{sec:inverse} and Lemma \ref{lem:Hausdorff} that $(\pmcg(S), \tau_w)$ is a topological group (see Proposition~\ref{prop:embedding}). 

The goal of the next key lemma (Lemma \ref{lem:closure}) is to detect Dehn twists in $\pmcg(S)$.
In the case of finite-type surfaces, the main idea behind Ivanov's results in \cite{IvanovAutomorphismsT} on the automorphisms of mapping class groups is to algebraically characterize Dehn twists.
Given an automorphism preserving Dehn twists, one can then build a map on the collection of simple closed curves in order to determine the isotopy class of a homeomorphism.
Building off the literature surrounding homomorphisms between mapping class groups, we take a similar approach in Section \ref{sec:automorphisms}.

Given a simple closed curve $a$ in $S$ let $[a]$ denote its homotopy class in $S$ and let $[a]_\lambda$ denote the homotopy class of $i_\lambda(a)$ in $S_\lambda$.
For the sake of the argument below, we record an obvious lemma: 

\begin{Lem}\label{lem:intersection}
For any $\lambda \in \Lambda$
\[
i([a],[b]) \geq i\left([a]_\lambda, [b]_\lambda\right),
\]
where $a$ and $b$ are any two simple closed curves in $S$.
\end{Lem}

Let $\sd$ denote the collection of Dehn twists in $\pmcg(S)$ and let $\sd_\lambda$ denote the collection of Dehn twists in $\pmcg(S_\lambda)$.
The subgroup of $\pmcg(S)$ generated by $\sd$ is the group of compactly supported mapping classes, denoted $\pmcg_c(S)$.
We will use this notation throughout the article.

\begin{Lem}\label{lem:closure}
The closure of $\sd$ in $\tau_w$ is $\overline \sd = \sd \cup \{id\}$. 
\end{Lem}

We note that this statement also holds in $\tau_q$; we give a proof in Section \ref{sec:cpt} (see Proposition~\ref{prop:closure2}).
Additionally, the reader should be warned that the subtlety in the proof of the lemma lies in the fact that the stabilizer of an isotopy class of a curve is not open in $\tau_w$.
This failure for a stabilizer of a curve to be open follows from the proof of Proposition \ref{prop:containment} and is demonstrated by the example illustrated in Figure \ref{fig:converge}, which shows an example of a (non-intuitive) sequence of Dehn twists converging to the identity.

\begin{proof}
Fix an exhaustion $K_1 \subset K_2 \subset\cdots$ of $S$ by essential finite-type $\genus(S)$-subsurfaces.
Observe, for later in the proof, that the direct limit of the groups $\{\pmcg(K_n)\}_{n\in \bn}$ is $\pmcg_c(S)$.

We first see that the identity is in the closure of $\sd$.
For each $K_n$ in our exhaustion choose an essential simple closed curve $b_n$ in $S$ contained in the complement of $K_n$.
The sequence $\{T_{b_n}\}_{n\in \bn}$ converges to the identity in $\tau_q$ as $T_{b_n}$ agrees with the identity on $K_n$ and $\{K_n\}_{n\in \bn}$ is an exhaustion.
The containment $\tau_w \subset \tau_q$ (see Proposition \ref{prop:containment}) guarantees this sequence also converges to  the identity in $\tau_w$. 

As $\tau_w$ is not first countable when $\Ends(S)$ is uncountable (see Lemma \ref{lem:first-countable}), it is necessary for the remainder of the proof for us to work with nets as opposed to sequences; however, the reader will lose no intuition by replacing the word net with sequence.

Let $I$ be a directed set and let $\{T_{c_i}\}_{i\in I}$ be a net in $\sd$, where $T_{c_i}$ denotes the Dehn twist about the simple closed curve $c_i$. 
Assume the net converges to $f \in (\pmcg(S), \tau_w)$.
Further assume that $f$ is not the identity, then we will show that $f$ is a Dehn twist.

We first claim that $f$ is an element of $\pmcg_c(S)$.
Suppose not, then, by Lemma \ref{lem:center}, $f$ is not in the center of $\pmcg(S)$  so there exists an essential simple closed curve $a$ with $[f(a)] \neq [a]$. 
Choose $N \in \bn$ such that both $a$ and $f(a)$ can be homotoped into $K_N$.
For each $n\in \bn$ with $n > N$ we can find $b_n$ homotopic into the complement of $K_n$ such that $[f(b_n)] \neq [b_n]$.
Fix $n \in \bn$ with $n > N$, then there exists $m \in \bn$ with $m > n$  and where $b_n$ and $f(b_n)$ are homotopic into $K_m$.
Choose $\lambda \in \Lambda$ such that $\lambda$ intersects each component of $\overline S \ssm K_m$. 

Given the choice of $\lambda$, we have $[f(a)]_\lambda \neq [a]_\lambda$ and $[f(b_n)]_\lambda \neq [b_n]_\lambda$.
It follows that $\vp_{\lambda}(f)$ is nontrivial and thus a Dehn twist, call it $T$, as our net is in $\sd$.
By construction, 
\[
T([a]_\lambda) = \vp_\lambda(f)([a]_\lambda) \neq [a]_\lambda
\]
and
\[
T([b_n]_\lambda) = \vp_\lambda(f)([b_n]_\lambda) \neq [b_n]_\lambda.
\]
Choose a simple closed curve $c$ in $K_m$ such that $T = T_{[c]_\lambda}$, then
\begin{itemize}
\item $i([c]_\lambda, [a]_\lambda)  > 0$,
\item $i([c]_\lambda, [b_n]_\lambda) > 0$, and
\item $i([a]_\lambda, [b_n]_\lambda) = 0$,
\end{itemize}
where the first two inequalities follow from the fact that $T$ fixes neither $[a]_\lambda$ nor $[b_n]_\lambda$. 
The last equality comes from the fact that $a$ and $b_n$ are disjoint combined with Lemma~\ref{lem:intersection}.
A direct application of \cite[Proposition 3.4]{Primer} gives the inequality
\[
i(T([a]_\lambda), [b_n]_\lambda) = i([c]_\lambda, [a]_\lambda)i([c]_\lambda, [b_n]_\lambda)> 0.
\]
Another application of Lemma \ref{lem:intersection} yields
\[
i([f(a)], [b_n]) \geq i([f(a)]_\lambda, [b_n]_\lambda) = i(T([a]_\lambda), [b_n]_\lambda) > 0.
\]
As $n$ was arbitrary, we see that $f(a)$ must exit every compact set of $S$; hence, it cannot be compact and $f$ is not a homeomorphism.

It follows there exists $n \in \bn$ such that $f\in \pmcg(K_n) < \pmcg_c(S)$.
Let $m \in \bn$ such that $m > n$ and $K_n$ is contained in the interior of $K_m$.
Choose $\lambda \in \Lambda$ such that $\lambda$ intersects each component of $\overline S\ssm K_m$ nontrivially.
We have that $\vp_\lambda$ restricted to $\pmcg(K_n)$ is injective and induced by the inclusion $ i_\lambda|_{K_n} \co K_n \hookrightarrow \overline{S}\ssm \lambda$; in particular, as $f \in \pmcg(K_n)$ and $\vp_\lambda(f)$ is a Dehn twist, we must have that $f$ is a Dehn twist.
\end{proof}

As direct consequence, we give a characterization of Dehn twists.
For $\lambda \in \Lambda$, let $\overline \sd_\lambda$ be the union of the Dehn twists in $\pmcg(S_\lambda)$ with the identity, so  $\overline \sd_\lambda = \vp_\lambda(\overline \sd)$. 

\begin{Prop}\label{prop:characterization}
$f\in \overline \sd$ if and only if $\vp_\lambda(f) \in \overline \sd_\lambda$ for every $\lambda \in \Lambda$. 
\end{Prop}

\begin{proof}
The forward direction is trivial.
For the reverse, assume $f\in \pmcg(S)$ such that $\vp_\lambda(f) \in \overline \sd_\lambda$ for every $\lambda \in \Lambda$.
For each $\lambda \in \Lambda$ choose $T_\lambda \in \overline \sd$ such that $\vp_\lambda(T_\lambda) = \vp_\lambda(f)$.
This implies that the net $\{T_\lambda\}_{\lambda \in \Lambda}$ converges to $f$; hence, $f \in \overline \sd$ by Lemma \ref{lem:closure}.
\end{proof}


\section{Residual Finiteness}
\label{sec:residual}

A group is \emph{residually finite} if the intersection of all of its finite-index proper normal subgroups is trivial.
As noted earlier, the work of Grossman \cite{GrossmanResidual} tells us that mapping class groups of finite-type surfaces are residually finite.
As this property is inherited by subgroups it follows that pure mapping class groups are also residually finite. 

In this section, we will prove Theorem \ref{thm:rf} in a series of propositions.
We break the section into two subsections handling the finite- and infinite-genus cases separately.

\subsection{Finite genus}
As a start, we properly record the result mentioned in Section \ref{sec:initial}:

\begin{Prop}
\label{prop:fgrf}
If $S$ has finite genus, then $\pmcg(S)$ is residually finite.
\end{Prop}

\begin{proof}
Let $f \in \pmcg(S)$ be an arbitrary nontrivial element, then by Lemma \ref{lem:Hausdorff} there exists $\lambda \in \Lambda$ such that $f \notin \ker \vp_\lambda$.
Further, since $\pmcg(S_\lambda)$ is residually finite \cite{GrossmanResidual} there exists a homomorphism $\psi\co \pmcg(S_\lambda) \to G$, where $G$ is a finite group and $f \notin \ker( \psi\circ \vp_\lambda)$.  
\end{proof}

In a slight change of perspective, we will focus on the full mapping class group $\mcg(S)$ for the remainder of this subsection.

Let $B_n$ denote the $n$-stranded braid group, which is also the mapping class group of the $n$-punctured disk.
The classical presentation for $B_n$ is given as follows:

\[
B_n = 
\Biggl\langle 
\si_1, \ldots, \si_{n-1} \left|
\begin{array}{ l}
	\si_i \si_j = \si_j \si_i : |i-j|\geq 2 \text { and } 1\leq i,j \leq n-1\\
	\si_i\si_{i+1}\si_i = \si_{i+1}\si_i \si_{i+1} : 1 \leq i \leq n-2
\end{array}
\right.\Biggr\rangle
\]

With this presentation we see that there are natural inclusions $\iota_{n,m}\co B_n \hookrightarrow B_{m}$ whenever $n \leq m$.
In particular $\langle B_n, \iota_{n,m}\rangle$ is a directed system of groups and we define $B_\infty$ to be the direct limit.
The presentation for $B_\infty$ is immediate from that of $B_n$, namely, 
\[
B_\infty = 
\Biggl\langle 
\{\si_i : i \in \bn\} \left|
\begin{array}{ l}
	\si_i \si_j = \si_j \si_i : |i-j|\geq 2 \text { and } i,j\in \bn\\
	\si_i\si_{i+1}\si_i = \si_{i+1}\si_i \si_{i+1} : i\in \bn
\end{array}
\right.\Biggr\rangle
\]

We will proceed by proving that $B_\infty$ is not residually finite and that it is a subgroup of every finite-genus big mapping class group.  

\begin{Prop}\label{prop:cyclic}
Every finitely-generated quotient of $B_\infty$ is cyclic.  
\end{Prop}

\begin{proof}
Let $f\co B_\infty \to G$ be a surjective homomorphism to a finitely-generated group.
We can then find some $N$ such that $f|_{B_N}$ is surjective.
For $m > N$,  $\si_m$ commutes with every element of $B_N$; hence, $f(\si_m)$ is in the center of $G$ for every $m > N$.
From the relations, it is easy to see that if $f(\si_i)$ commutes with $f(\si_{i+1})$, then $f(\si_i) = f(\si_{i+1})$.
It now follows that $G$ is generated by $f(\si_1)$.  
\end{proof}

The abelianization of $B_\infty$ is isomorphic to $\bz$; in particular, it follows from Proposition~\ref{prop:cyclic} that every finite quotient of $B_\infty$ factors through its abelianization.
Thus the commutator subgroup is in the intersection of every proper finite-index normal subgroup of $B_\infty$.
This shows:

\begin{Cor}\label{cor:braid}
$B_\infty$ is not residually finite.
\end{Cor}

\begin{Prop}\label{prop:braid}
Let $S$ be an infinite-type surface.  
If $S$ has finite genus, then $B_\infty$ is a subgroup of $\mcg(S)$.
\end{Prop}

In fact, Proposition \ref{prop:braid} holds whenever $S$ has infinitely many planar ends, but we will not use the full generality of the result in what follows.

\begin{proof}

The proof splits into two cases: either $\Ends(S)$ has infinitely or finitely many isolated points.
Let us start by assuming that $\Ends(S)$ has infinitely many isolated points.
Together, the Cantor-Bendixson Theorem and Brouwer's topological characterization of the Cantor set imply that $\Ends(S) = C \sqcup P$ (the disjoint union is as sets, not as topological spaces), where $C$ is a Cantor space and $P$ is countable (see \cite[Theorem 6.4 \& Theorem 7.4]{KechrisClassical}).
It follows that the isolated ends of $S$ are denumerable, so we can choose a labelling $\{p_1, p_2, \ldots\}$ of the isolated ends.
For each $n\in \bn$ let $D_n \subset S$ be an $n$-punctured disk satisfying:
\begin{enumerate}
\item
$\Ends(D_n) = \{p_1, \ldots, p_n\}$, and
\item
$D_{n}\subset D_{n+1}$.
\end{enumerate}
It follows that $D_n$ is an essential subsurface of $S$; in particular, the embedding $D_n \hookrightarrow S$ induces an injection $\mcg(D_n) \into \mcg(S)$. 
Recall that $\mcg(D_n) \cong B_n$.
By the universal property of direct limits, the inclusions $D_{n} \hookrightarrow D_{n+1}$ induce an injective homomorphism $B_\infty \hookrightarrow \mcg(S)$ as desired.

Now assume $\Ends(S)$ has a finite number of isolated points.
Again applying the Cantor-Bendixon Theorem, we have  $\Ends(S) = C \sqcup F$, where $C$ is a Cantor space and $F$ is the finite collection of isolated points.
Let $D \subset \overline S$ be a disk such that \( D \cap F = \emptyset \) and \( D \cap C = C' \) where \( C' \) is a clopen subset of \( C \) and hence a Cantor space.
The work of Funar-Kapoudjian \cite[Section 7]{FunarUniversal} implies $B_\infty < \mcg(D\ssm C')$; thus, $B_\infty < \mcg(S)$ as $D \ssm C'$ is an essential subsurface.
\end{proof}

Combining Corollary \ref{cor:braid} and Proposition \ref{prop:braid}, we have:

\begin{Cor}
If $S$ is a finite-genus surface of infinite type, then $\mcg(S)$ is not residually finite.
\end{Cor}

\subsection{Infinite genus}

In this subsection we will prove the remaining pieces of Theorem \ref{thm:rf} involving infinite-genus surfaces.

\begin{Prop}\label{prop:no}
If $S$ has infinite genus, then $\pmcg_c(S)$ has no finite quotients.
\end{Prop}

This has an immediate corollary:

\begin{Cor}
$\pmcg_c(S)$ is perfect and does not virtually surject onto $\bz$.
\end{Cor}

Mapping class groups of finite-type surfaces are known to be perfect \cite{PowellTwo} (except for some low-complexity cases).
Motivated by Kazhdan's Property (T), Ivanov conjectured  that mapping class groups of finite-type surfaces do not virtually surject onto $\bz$ (see \cite[Problem 2.11.A]{KirbyProblems}).
(Equivalently, this says $H_1(G, \mathbb Q) = 0$ for any finite-index subgroup $G$ of $\mcg(S)$.)
Putman-Wieland \cite{PutmanAbelian} have proven stability conditions for this conjecture; it would be fascinating if the stability could manifest into a statement about big mapping class groups. 

As a second  corollary, we have our desired result in relation to Theorem \ref{thm:rf}: 

\begin{Cor}
\label{cor:igrf}
Neither $\mcg(S)$ nor $\pmcg(S)$ is residually finite when $S$ has infinite genus.
\end{Cor}

The heavy lifting in the proof of Proposition \ref{prop:no} is done by a theorem of Paris \cite{ParisSmall}, which says that when $S$ is a finite-type genus-$g$ surface, the minimal index of a proper subgroup of $\pmcg(S)$ is bounded below by a linear function in $g$.
(The work of Berrick-Gebhardt-Paris \cite{ParisFinite} gives the exact value as an exponential function in $g$.)

\begin{proof}[Proof of Proposition \ref{prop:no}]
Let $\{F_n\}_{n\in\bn}$ be an exhaustion of $S$ by essential finite-type surfaces, then $\genus(F_n) \to \infty$ as $n\to \infty$.
Let $m_n$ denote the minimal index of a subgroup in $\pmcg(F_n)$, then $\lim_{n \rightarrow \infty} m_n = \infty$ \cite{ParisSmall}.  

Suppose $N$ is a proper normal subgroup of $\pmcg_c(S)$ with finite index $k$.
As $\pmcg_c(S)$ is the direct limit of $\{\pmcg(K_n)\}_{n\in \bn}$, we have that there exists $M \in \bn$ such that $\pmcg(K_m)$ is not a subgroup of $N$ for all $m > M$ (otherwise $N$ would not be proper).
%
It follows that for $m > M$ the intersection $N \cap \pmcg(K_m)$ is a proper normal subgroup of $\pmcg(K_m)$ with index at most $k$.
However, this implies that the sequence $\{m_n\}_{n\in\bn}$ is bounded, a contradiction.
\end{proof}


\section{Morphisms}
\label{sec:automorphisms}

The goal of this section is to finish the proof of Theorem \ref{thm:main} with Propositions \ref{thm:genus} and \ref{prop:homeo}.
We then deduce corollaries about the outer automorphism group.
The tools in this section combine the results on the initial topology in Section \ref{sec:initial} with the literature on homomorphisms between mapping class groups (namely, \cite{AramayonaHomomorphisms, Castel}).
By Proposition \ref{prop:fgrf} and Corollary \ref{cor:igrf}, the algebraic structure of $\pmcg(S)$ detects whether $S$ has finite or infinite genus.
For the remainder of the section, we focus on finite-genus surfaces.

\begin{Prop}
\label{thm:genus}
Let $S$ and $S'$ be finite-genus surfaces.
If the genus of $S$ is at least 3 and strictly greater than  that of $S'$, then there is no monomorphism from $\pmcg(S)$ to $\pmcg(S')$.
\end{Prop}

\begin{proof}
Suppose $S$ and $S'$ are surfaces satisfying 
\[\genus(S) > \max\{3, \genus(S')\}\]
and that $\psi \co \pmcg(S) \to \pmcg(S')$ is a homomorphism.
Let $F \subset S$ be an essential finite-type subsurface with $\genus(F) = \genus(S)$ and $\iota\co \pmcg(F) \hookrightarrow \pmcg(S)$ the inclusion homomorphism.  
For each $\lambda \in \Lambda(S')$, we obtain a homomorphism $\psi_\lambda \co \pmcg(F) \to \pmcg(S'_\lambda)$ defined by the composition $\psi_\lambda = \vp_\lambda \circ \psi \circ \iota$.
As $F$ and $S'_\lambda$ are finite-type and $\genus(F) > \max\{3, \genus(S_\lambda')\}$, it follows from \cite[Proposition 7.1]{AramayonaHomomorphisms} that $\psi_\lambda$ is trivial.
Further, as $\lambda$ was arbitrary, we see that 
\[
\psi\circ \iota (\pmcg(F)) \subset \bigcap_{\lambda \in \Lambda(S')}\ker \vp_\lambda
\]
implying that $ \iota(\pmcg(F))\subset \ker \psi$ by Lemma \ref{lem:Hausdorff}; hence, $\psi$ is not injective.
\end{proof}

Continuing with our progress towards Theorem \ref{thm:main}, we can now assume that if $\pmcg(S)$ and $\pmcg(S')$ are isomorphic, then they have the same genus.  We have one additional lemma before finishing the final proposition towards Theorem \ref{thm:main}. 

\begin{Lem}
\label{lem:preserve}
Let $S$ and $S'$ be surfaces such that their genera are finite, equal, and at least 4.
If $\Psi\co \pmcg(S) \to \pmcg(S')$ is an isomorphism, then $\Psi$ preserves Dehn twists.
\end{Lem}

\begin{proof}
Let $T_c$ denote the Dehn twist along a simple closed curve $c$.
Choose an essential finite-type surface $F\subset S$ such that $c\subset F$ and $\genus(F) = \genus(S)$.
The inclusion $F \hookrightarrow S$ induces a homomorphism $\iota\co\pmcg(F) \to \pmcg(S)$. 
For each $\lambda \in \Lambda(S')$ we define $\iota_\lambda\co \pmcg(F) \to \pmcg(S'_\lambda)$ to be the composition $\vp_\lambda \circ \Psi \circ \iota$.
If \( \iota_\lambda \) is nontrivial, then \cite[Theorem~1.1]{AramayonaHomomorphisms} implies $\iota_\lambda$ is induced by an inclusion $F \hookrightarrow S'_\lambda$ (note that this is where the lower bound on genus is required\footnote{The statement of \cite[Theorem~1.1]{AramayonaHomomorphisms} requires the genus to be at least 6; however, the remark immediately following the theorem statement in the article states that this requirement can be decreased to 4 under our assumptions.}).

In particular, $\iota_\lambda$ sends Dehn twists in $\pmcg(F)$ into $\overline \sd_\lambda$.
As $T_c$ is the image of a Dehn twist in $\pmcg(F)$ under $\iota$, we have that $\vp_\lambda \circ \Psi (T_c) \in \overline \sd_\lambda$ for each $\lambda \in \Lambda(S')$. 
By Proposition \ref{prop:characterization}, $\Psi(T_c)$ is a Dehn twist.
\end{proof}

\begin{Prop}
\label{prop:homeo}
Let $S$ and $S'$ be surfaces.
If the genera of $S$ and $S'$ are finite, equal, and at least 4, then any isomorphism between $\pmcg(S)$ and $\pmcg(S')$ is induced by a homeomorphism.
\end{Prop}

\begin{proof}
Let $\Psi \co \pmcg(S) \to \pmcg(S')$ be an isomorphism.
Observe that the centers $Z(\pmcg(S))$ and $Z(\pmcg(S'))$ of $\pmcg(S)$ and $\pmcg(S')$, respectively, are isomorphic, and therefore, $S$ and $S'$ have the same number of boundary components by Lemma \ref{lem:center}. 
By applying Lemma \ref{lem:preserve} to both $\Psi$ and $\Psi^{-1}$, we see that $\Psi$ restricts to an isomorphism from $\pmcg_c(S)$ to $\pmcg_c(S')$.

Let $F_1 \subset F_2 \subset \cdots$ be an exhaustion of $S$ by essential finite-type surfaces such $\genus(F_n) = \genus(S)$  and \( \partial S \subset \partial F_n \) for each $n\in \bn$.
Recall that essential surfaces are connected by definition. 
It follows that the direct limit of the groups $\{\pmcg(F_n)\}_{n\in \bn}$ with respect to the natural inclusions is $\pmcg_c(S)$.

Fix \( n \in \bn \) and let \( d_1, \ldots, d_m \) be collection of essential simple closed curves in \( F_n \) such that the set \( \{ T_{d_1}, \ldots, T_{d_m}\} \) of Dehn twists generates \( \pmcg(F_n) \).
For each \( i \in \{ 1, \ldots, m \} \) let \( d_i' \) be the simple closed curve in \( S' \) such that \( \Psi(T_{d_i}) = T_{d_i'} \); the existence of \( d_i' \) is guaranteed by Lemma \ref{lem:preserve}.
Let \( R' \) be a regular neighborhood of \( d_1' \cup \cdots \cup d_m' \).
It is possible that one or more boundary components of \( R' \) bounds either a disk or once-punctured disk in \( S' \).
If this is the case, glue the associated disk or once-punctured disk onto these boundary components; we call the resulting surface \( R_n \).


Observe that since any Dehn twist about a boundary component of \( S' \) can be written as a product of Dehn twists about curves contained in \( R_n \), we can conclude that \( \partial S' \subset \partial R_n \).
We now claim that the surface \( R_n \) is connected:  Indeed, observe that \( d_i \) and \( d_j \) are disjoint if and only if \( d_i' \) and \( d_j' \) are disjoint.
Using that \( \Psi \) is an isomorphism, the above fact follows from the fact that \( T_{d_i} \) and \( T_{d_j} \) commute if and only if \( d_i \) and \( d_j \) are disjoint.
It now follows that \( d_1' \cup \cdots \cup d_m' \) is connected as \( d_1 \cup \cdots \cup d_m \) is connected; hence, \( R_n \) is connected.

As $\genus(F_n) = \genus(S) = \genus(S')$, we have (again by \cite[Proposition~7.1]{AramayonaHomomorphisms}) that $\genus(R_n) = \genus(F_n)$.
Now if \( f \in \pmcg(S') \) is in the centralizer of \( \pmcg(R_n) \), then \( f \) commutes with \( T_{d_i'} \) for each \( i \in \{1, \ldots, m\} \); in particular, \( f \) is in the centralizer of \( \Psi(\pmcg(F_n)) \).
Conversely, if \( f \in \pmcg(S') \) is in the centralizer of \( \Psi(\pmcg(F_n)) \), then again \( f \) commutes with the Dehn twist \( T_{d_i'} \) for every \( i \in \{1, \ldots, m\} \); therefore, as the collection of curves \( \{d_1', \ldots, d_m'\} \) is filling in \( R_n \), we can conclude that \( f \) is in the centralizer of \( \pmcg(R_n) \).
The above discussion also implies that an element in the centralizer of \( \pmcg(R_n) \) is supported in the complement of the interior of \( R_n \); therefore, by Lemma \ref{lem:center}, the center of the centralizer of \( \pmcg(R_n) \) is generated by the boundary components of \( R_n \).
In particular, the rank of the center of the centralizer of \( \pmcg(R_n) \) is the number of boundary components of \( R_n \) (the same holds for \( F_n \) and \( \pmcg(F_n) \)). 
Now the center of the centralizer of \( \pmcg(F_n) \) and that of \( \Psi(\pmcg(F_n)) \) are isomorphic; it follows from the above that \( F_n \) and \( R_n \) have the same number of boundary components.


Let $Z_n$ denote the center of $\pmcg(F_n)$.
From the preceding discussion, we have established that the center of the centralizer of \( \pmcg(R_n) \) is the center of \( \pmcg(R_n) \).
The same holds for \( \pmcg(F_n) \) in \( \pmcg(S) \). 
Further, we can conclude that \( \Psi(Z_n) \) is the center of \( \pmcg(R_n) \).
This implies that $\Psi$ induces a monomorphism from $\pmcg(F_n)/Z_n$ to $\pmcg(R_n)/\Psi(Z_n)$. 
Observe that $\pmcg(F_n)/Z_n$ is the pure mapping class group of the surface obtained by capping off the boundary components of $F_n$ with once-punctured disks and the same is true for $\pmcg(R_n)/\Psi(Z_n)$ and $R_n$.
It follows from \cite[Theorem 1.2.7]{Castel} plus that fact that $F_n$ and $R_n$ have the same number of boundary components that they also have the same number of punctures; in particular, $F_n$ and $R_n$ are homeomorphic.

We can now apply \cite[Theorem 1.1]{AramayonaHomomorphisms} to see that $\Psi_n :=\Psi|_{\pmcg(F_n)}$ is induced by an embedding $\psi_n\co F_n \hookrightarrow R_n$; it follows from our discussion thus far that $\psi_n$ is in fact a homeomorphism and $\Psi_n$ an isomorphism.
As $\Psi|_{\pmcg_c(S)}$ is an isomorphism to $\pmcg_c(S')$, it follows that $\pmcg_c(S')$ is the direct limit of groups $\{\pmcg(R_n)\}_{n\in \bn}$.
It also follows that \(R_1 \subset R_2 \subset \cdots\) is an exhaustion of $S'$ by essential finite-type surfaces: indeed, if not, then there an essential simple closed curve in the complement of \( \bigcup_{n\in \bn} R_n \), but this would imply that the Dehn twist about this curve is not an element of \( \bigcup_{n\in \bn} \pmcg(R_n) \), a contradiction.
We can now conclude that $S'$ is the direct limit of the spaces $\{R_n\}_{n\in\bn}$ with respect to inclusion.
Note that $S$ is the direct limit of the spaces $\{F_n\}_{n\in\bn}$ as well. 
Given the construction of $\psi_n$, we can build the following commutative diagram where the indexed $i$ and $j$'s are the natural inclusions and the natural numbers $n$ and $m$ satisfy $m>n$:
\[
\xymatrix{
& F_n \ar[dl]^{\psi_n} \ar[rr]_{i_{m,n}} \ar[dr]_{i_n} & & F_m \ar[dr]_{\psi_m}  \ar[dl]^{i_m} &  \\
R_n \ar[ddrr]^{j_n} & & S \ar@{.>}[dd]|-\psi & & R_m \ar[ddll]_{j_m} \\
&&&& \\
& & S' & &
}
\]
The existence of the unique continuous map $\psi\co S \to S'$ is guaranteed by the universal property of direct limits.  
If we repeat the above process starting with $\Psi^{-1}$, we see that $\psi$ is the desired homeomorphism between $S$ and $S'$. 
\end{proof}

Observe that after the immediate application of Lemma \ref{lem:preserve} in the proof  Proposition \ref{prop:homeo}, the proof reduces to proving the analogous statement for $\pmcg_c(S)$.  
The inspiration for this portion comes from \cite[Proposition 9.1]{AramayonaAsymptotic} in which the authors show that an automorphism of $\pmcg_c(S)$ -- in the special case where $S$ has finite genus and $\Ends(S)$ is a Cantor set -- preserves Dehn twists.

As pointed out in the introduction, Corollary \ref{cor:automorphism} is an immediate consequence of Proposition \ref{prop:homeo} and Lemma \ref{lem:center}, that is, for a borderless surface \( S \) we have
\[
\Aut(\pmcg(S)) = \mcg^\pm(S)
\]
when $4\leq \genus(S) < \infty$.
We point out an alternative proof:
By Lemma \ref{lem:preserve}, an automorphism permutes the Dehn twists and hence induces an automorphism of the curve graph (two curves are disjoint if and only if their associated Dehn twists commute).
In an recent paper \cite{Hernandez2}, it is shown that the automorphism group of the curve graph is the extended mapping class group.
One can then follow a -- by now -- standard argument to obtain Corollary \ref{cor:automorphism}.
(This is the route taken in the recent work \cite{BavardIsomorphism} in which Question \ref{ques:mainquestion} is answered.)

As a corollary to Corollary \ref{cor:automorphism}, we see the topological structure of $\Ends(S)$ becoming apparent in the algebra of $\pmcg(S)$.  

\begin{Cor}
\label{cor:ends}
If $S$ is a borderless surface whose genus is finite and  at least 4, then $\out(\pmcg(S))$  is isomorphic to $\homeo(\Ends(S)) \times \bz/2\bz$.
\end{Cor}

\begin{proof}
Corollary \ref{cor:extension} tells us the homomorphism $\mcg(S) \to \homeo(\Ends(S))$ is surjective and by extension $\mcg^\pm(S) \to \homeo(\Ends(S))$ is also surjective.
It follows that there exists an order-two orientation-reversing homeomorphism fixing all the ends of $S$; let $f \in \mcg^\pm(S)$ be such an element.
It follows that
\[
\mcg^\pm(S) = \mcg(S) \rtimes \bz/2\bz,
\]
where the $\bz/2\bz$ factor is identified with $\langle f \rangle$.
It easily follows from Corollary \ref{cor:automorphism}  and Corollary \ref{cor:extension} that
\(
\out(\pmcg(S)) = \mcg^\pm(S) / \pmcg(S)  = \homeo(\Ends(S)) \times \bz/2\bz.
\)
\end{proof}

Given Corollary \ref{cor:ends}, it is natural to ask if the isomorphism type of $\out(\pmcg(S))$ determines the homeomorphism type of $\Ends(S)$.
This is equivalent to asking: given a closed subset $E$ of the Cantor set, does the isomorphism type of $\homeo(E)$ determine the homeomorphism type of $E$?

Note that there is an easy counterexample:  if $E = C \sqcup X$ where $X$ is a one-point space and $C$ is a Cantor space, then $\homeo(C) \cong \homeo(E)$.
It was originally conjectured by Monk \cite{MonkAutomorphism} that this was the only counterexample; however, another counterexample was given by McKenzie \cite[Theorem 6]{McKenzieAutomorphism}.
On the other hand, McKenzie \cite[Theorem 5]{McKenzieAutomorphism} showed the conjecture holds in a large subclass, namely if the closure of the isolated points in $E$ is clopen, then $\homeo(E)$ determines $E$.
(The work of Monk and McKenzie is in the category of boolean algebras; one can translate to the language above using Stone's representation theorem.)

In light of Theorem \ref{thm:main}, it is natural to pose the following:

\begin{Problem}
Find correspondences between algebraic invariants of $\pmcg(S)$ and topological invariants of $S$.
\end{Problem}

As an example, we highlight the work of Monk:

\begin{Prop}[{\cite[Corollary 2.1, Theorem 3, \& Theorem 5]{MonkAutomorphism}}]
Let $E$ be a closed subset of the Cantor set with infinite cardinality.
\begin{enumerate}
\item $\homeo(E)$ is simple if and only if $E$ is homeomorphic to either a Cantor space or the disjoint union of a Cantor space and the one-point space.
\item If $m \in \bn$ with $1<m<\infty$, then $E$ is homeomorphic to $C \sqcup \{p_1, \ldots, p_m \}$ if and only if $\homeo(E) \cong \mathrm{Sym}_m \times \homeo(C)$, where $C$ is a Cantor space. 
\item $\homeo(E)$ has exactly two proper normal subgroups if and only if $E$ is homeomorphic to the one-point compactification of $\bn$.
\end{enumerate}
\end{Prop}


\section{The compact-open topology}\label{sec:cpt}

We begin by restating the definition of the mapping class group: Equip $\homeo(S)$ with the compact-open topology, then
\[
\mcg(S) = \pi_0(\homeo^+(S, \partial S)),
\]
where \(\homeo^+(S, \partial S)\) is the group of orientation-preserving homeomorphisms of $S$ fixing $\partial S$ pointwise.
With this definition, $\mcg(S)$ comes equipped with a natural topology, namely the corresponding quotient topology, giving $\mcg(S)$ the structure of a topological group. 
If $S$ is a finite-type surface, then this is the discrete topology on $\mcg(S)$ (one way to see this is as a consequence of the Alexander method \cite[Proposition 2.8]{Primer}).
We set $\tau_q$ to be the corresponding subspace topology on $\pmcg(S)$.
It follows from Lemma \ref{lem:hernandez} that the identity component of \( \homeo^+(S) \) is closed when \( S \) is borderless.
This implies that the one-point sets in \( (\pmcg(S), \tau_q) \) are closed; it is a standard exercise to show that in a topological group this implies the group is Hausdorff.
(A similar argument works for surfaces with boundary.)

Moreover, it is metrizable:

\begin{Prop}
\label{prop:metrizable}
$(\pmcg(S), \tau_q)$ is metrizable.
\end{Prop}

\begin{proof}
As $\homeo(S, \partial S)$ is first countable and $\mcg^\pm(S)$ is the image under an open quotient map, it follows that $(\pmcg(S), \tau_q)$ is first countable.
By the Birkhoff-Kakutani Theorem \cite{Birkhoff, Kakutani}, every first-countable Hausdorff topological group is metrizable.  
\end{proof}

Recall that for a finite-type surface, a combination of the Dehn-Lickorish theorem and the Birman exact sequence (see \cite[Chapter 4]{Primer}) tells us that $\pmcg(S)$ is generated by Dehn twists. 
In the infinite-type setting this of course cannot be true as $\pmcg(S)$ is uncountable; however, building off the result in the finite-type setting we are able to obtain topological generating sets for $\pmcg(S)$.
The remainder of this section is dedicated to proving Theorem \ref{thm:dense}.

Before we get to the proof of Theorem \ref{thm:dense}, we need to introduce a class of homeomorphisms contained in $\pmcg(S)$.
Consider the surface $\Sigma$ defined by taking the surface $\br \times [0,1]$, removing the interior of each disk of radius $\frac14$ with center of the form $(n, \frac12)$ for each $n \in \bz$, and attaching a torus with one boundary component to the boundary of each such disk. 
Let $\si$ be the homeomorphism of $\Sigma$  that behaves like $(x,y) \mapsto (x+1, y)$ on the interior of $\Sigma$ and tapers to the identity in a neighborhood of $\partial \Sigma$ (see Figure \ref{fig:genus-shift}).
One should think of $\si$ as obtained by sliding the $n^{th}$-handle (in the disk with center $(n, \frac12)$) horizontally until it reaches the position of the $(n+1)^{th}$-handle.

\begin{figure}
\center
\includegraphics{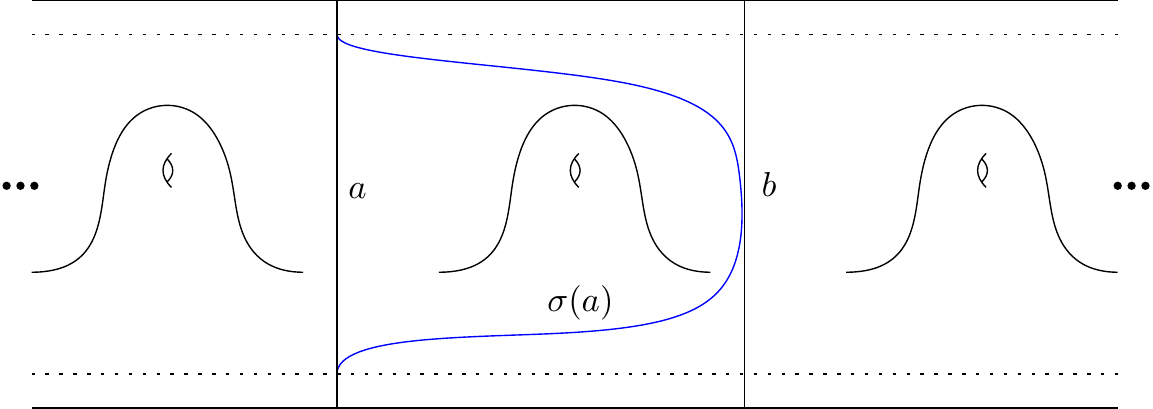}
\caption{Pictured here is the surface $\Sigma$.  The arc $\si(a)$ is the union of the blue arc and the black arc outside the dotted lines.  Inside $\Sigma$, the arcs $\si(a)$ and $b$ are homotopic.}
\label{fig:genus-shift}
\end{figure}

Given a surface $S$ with at least two ends accumulated by genus, consider an embedding of $\Sigma$ into $S$ inducing an injection $\Ends(\Sigma) \hookrightarrow \Ends(S)$. 
The homeomorphism $\si$ can then be extended to all of $S$ by the identity.
The homotopy class of the extension $\si$ defines an element of $\pmcg(S)$; we call such an element a \emph{handle shift}.

Towards Theorem \ref{thm:dense} we prove the following proposition. 

\begin{Prop}\label{prop:topgen}
The set of Dehn twists topologically generate $(\pmcg(S), \tau_q)$ when $S$ has at most one end accumulated by genus. If $S$ has at least two ends accumulated by genus, then the set of Dehn twists together with the set of handle shifts topologically generate $(\pmcg(S), \tau_q)$. 

\end{Prop}

\begin{proof}
Let $\sd$ denote the collection of Dehn twists in $\pmcg(S)$. 
Since Dehn twists generate finite-type pure mapping class groups and $\pmcg_c(S)$ is a direct limit of finite-type pure mapping class groups, it follows that $\langle \sd \rangle = \pmcg_c(S)$.
In light of the finite-type results, we will assume $S$ is of infinite type. When $S$ has at least two ends accumulated by genus, let $\pmcg_h(S)$ denote the subgroup of $\pmcg(S)$ generated by Dehn twists and handle shifts.

Take $f \in \pmcg(S)$ to be an arbitrary element. 
Abusing notation, we will conflate $f$ with a representative homeomorphism.
Fix a pants decomposition $\{P_1, P_2, \ldots \}$ of $S$ labelled so that the surface $R_n = \bigcup_{i=1}^n P_i$ is connected. 
Choose a component, call it $a$, of $\partial P_1$ and let $F$ be a finite-type essential surface such that 

\begin{itemize}
\item both $a$ and $f(a)$ are contained in $F$,
\item each component of $\partial F$ is separating, and
\item each component of $\overline S \ssm F$ intersects $\Ends(S)$ nontrivially.
\end{itemize}

When $S$ has at most one end accumulated by genus, we additionally require $F$ to be such that 

\begin{itemize}
\item at most one component of $ S\ssm F$ has positive genus.
\end{itemize}

By the classification of surfaces, the $\pmcg(S)$-orbit of $a$, up to homotopy, is determined by its partition of $\Ends(S)$, its partition of $\partial S$, and the topological type of its complements.
(This is true for any surface and any simple closed curve in the surface.) 
Observe that the construction of $F$ guarantees that $a$ and $f(a)$ must induce the same partition of $\partial F$ and $\Ends(F)$ as well.

First, if $a$ is nonseparating, then both $a$ and $f(a)$ are nonseparating in $F$, so there exists $g_a \in\pmcg(F) <  \pmcg(S)$ such that $g_a(a) = f(a)$.
In addition, we require that \( g_a \) fixes an orientation on \( a \).

Now suppose $a$ is separating. In the case where $S$ has at most one end accumulated by genus, there is a component $U$ of $S\ssm a$ that has finite genus.
Let $W = U \cap F$ and $V = f(U) \cap F$.
By the construction of $F$, it follows that $\genus(W) = \genus(V)$, and therefore, the components of $F\ssm a$ and $F \ssm f(a)$ have the same topological type. 
In particular, there exists $g_a \in \pmcg(F) < \pmcg(S)$ such that $g_a(a) = f(a)$. In both of these cases, $g_a$ is supported on a finite-type surface and hence is contained in $\pmcg_c(S)$. 

Since $g_a^{-1} \circ f$ fixes $a$ and both $f$ and $g_a$ fix $\Ends(S)$ pointwise, $g_a^{-1}\circ f$ restricts to a homeomorphism of each component of $S\ssm a$.
Let $b \neq a$ be another component of $\partial P_1$, and let $S'$ be the component of $S\ssm a$ containing $b$.
We can then repeat the above process with $S$ replaced by $S'$ and $f$ by $g_a^{-1} \circ f|_{S'}$.
If $P_1$ has a third boundary component, then we repeat this process with this final  component.
In either case, we have built a homeomorphism $g_1 \in \pmcg_c(S)$ such that $g_1(P_1) = f(P_1)$ and $g_1^{-1}\circ f$ fixes each boundary component of $P_1$.
Thus, $f_1 = g_1^{-1}\circ f$ restricts to the identity on $P_1 = R_1$.

Let $S_2$ be the component of $S\ssm P_1$ containing $P_2$.
Following the same reasoning as above and working in $\pmcg(S_2)$, we can find a product of Dehn twists in $\pmcg(S_2)$, call it $g_2'$, such that $g_2'^{-1}\circ f_1'$ is the identity when restricted to $P_2$, where $f_1'$ is the restriction of $f_1$ to $S_2$.
Choose an element $g_2''$ in the stabilizer of $P_1$ in $\pmcg_c(S)$ whose restriction to $S_2$ agrees with $g_2'$. 
Defining $g_2= g_1 \circ g_2'' \in \pmcg_c(S)$ we have that $g_2^{-1}\circ f$ restricts to the identity on $R_2$ after composing with Dehn twists about the components of $\partial P_1$ if necessary.

Continuing in this fashion, we construct $g_n \in \pmcg_c(S)$ such that $g_n^{-1} \circ f$ restricted to $R_n$ is the identity.
As $R_n$ is an exhaustion of $S$, we see that $g_n^{-1}\circ f$ converges to the identity; hence, $g_n \to f$ as desired.

Now suppose that $S$ has at least two ends accumulated by genus.
Note that we may no longer be able to choose $F$ to satisfy the fourth bullet point above. 
Let $a$ be a separating curve with each component of $S\ssm a$ being infinite genus and let $U$ be a component of $F \ssm a$. 
It is possible that $\genus(V) \neq \genus(W)$, where $W=U \cap F$ and $V=f(U) \cap F$.
If this is the case, there exists a handle shift $h \in \pmcg_h(S)$ where $h(a)$ induces the same partition of $\partial F$ and $\Ends(F)$ as $a$ and $f(a)$, where $h(a) \subset F$, and such that the genus of $h(U) \cap F$ and $V$ are equal.
Thus, there exists $g_a \in \pmcg(F) < \pmcg(S)$ such that $g_a \circ h(a) = f(a)$. 
Of course, $g_a \circ h\in \pmcg_h(S)$, so we can now repeat the argument above to produce $k_n \in \pmcg_h(S)$ such that $k_n \to f$.
\end{proof}

The next proposition shows that introducing handle shifts is essential when $S$ has at least two ends accumulated by genus. 

\begin{Prop}\label{prop:handle}
If $S$ has at least two ends accumulated by genus, then $\pmcg_c(S)$ is not dense in $(\pmcg(S), \tau_q)$. 
\end{Prop}

\begin{proof}
We begin by describing a property of all elements of $\pmcg_c(S)$, and then show that handle shifts do not have this property and cannot be approximated by those elements that do.
Choose a separating simple closed curve $a$ such that each component of $S\ssm a$ is infinite genus.  
Let $f$ be an arbitrary element of $\pmcg_c(S)$.
Let $F$ be an essential finite-type subsurface of $S$ so that $f \in \pmcg(F) < \pmcg(S)$ and $F$ satisfies the first three bullet points in the proof of Proposition \ref{prop:topgen}.
Let $U$ be a component of $F\ssm a$, then $U$ and $f(U)$ contain the same components of $\partial F$.
Further, $\genus(U) = \genus(f(U))$, so either $[f(a)] = [a]$ or $i([a],[f(a)]) > 0$.

Now, let $b$ be a simple closed curve so that $a$ and $b$ co-bound a genus--1 subsurface of $S$.
It is not hard to see that there exists a handle shift $h\in \pmcg(S)$ with $h(a) = b$. 
Therefore, no element in $\pmcg_c(S)$ can agree with $h$ on the compact set $a$, since $b = h(a)$ and $a$ are distinct and disjoint; hence, $\pmcg_c(S)$ is not dense in $(\pmcg(S), \tau_q)$.
\end{proof}

Propositions \ref{prop:topgen} and \ref{prop:handle} together imply Theorem \ref{thm:dense}.

In Lemma \ref{lem:closure}, we gave the closure of the set of Dehn twists in $\tau_w$.
In the finite-genus setting, the analogous result holds for $\tau_q$ and is obtained by combining Lemma \ref{lem:closure} and Proposition \ref{prop:containment}.
Below we give a proof in general, regardless of whether the genus of $S$ is finite or infinite:

\begin{Prop}
\label{prop:closure2}
The closure of $\sd$ in $\tau_q$ is $\overline \sd = \sd \cup \{id\}$.
\end{Prop}

\begin{proof}
Let $K_1 \subset K_2 \subset \cdots$ be an exhaustion of $S$ by essential finite-type subsurfaces. 
We first see that the identity is in $\overline \sd$ (which is a repeated argument from the beginning of Lemma \ref{lem:closure}). 
Let $b_n$ be a simple closed curve contained in the complement of $K_n$, then the sequence $\{T_{b_n}\}_{n\in\bn} \subset \sd$ converges to the identity in $(\pmcg(S), \tau_q)$.

Following the proof outline of Lemma \ref{lem:closure}, we want to show that if a sequence $\{T_{a_i}\}_{i\in\bn} \subset \sd$ converges to a non-identity element $f\in (\pmcg(S), \tau_q)$, then $f$ is compactly supported.

Assume that $f \notin \pmcg_c(S)$.
It follows that there exists a natural number $N$ such that for every integer $n$ with $n \geq N$ there exists an integer $i_{n}$ such  that the curve $a_{i_{n}}$, up to homotopy, is contained in neither $K_n$ nor its complement and intersects $K_N$ nontrivially.

By possibly increasing $i_{n}$, we may assume that the curve $a_{i_{n}}$ has nontrivial intersection with a single boundary component of $\partial K_N$, call it $b$.
For each curve $a_{i_{n}}$, let $b_{n}$ be a component of $K_n$ such that $i(a_{i_{n}}, b_{n}) > 0$.
As the compact-open topology agrees with the topology of compact convergence, there exists $J\in \bn$ such that, up to homotopy, $f(b) = T_{a_j}(b)$ for all $j > J$.
Apply \cite[Proposition 3.4]{Primer} to obtain
\[
i(f(b), b_n) = i(T_{a_{i_n}}(b),b_n) = i(a_{i_n},b)i(a_{i_n},b_n) > 0,
\]
whenever $i_n> \max\{J,N\}$.
It follows that $f(b)$ leaves every compact set; hence, it must not be compact and $f$ cannot be a homeomorphism.
Therefore, $f \in \pmcg_c(S)$.

We can now find $m \in \bn$ such that all but finitely many of the $T_{a_i}$ are elements of  $\pmcg(K_m) < \pmcg_c(S)$.
The subspace topology on $\pmcg(K_m)$ is discrete, so the sequence $\{T_{a_i}\}_{i\in \bn}$ is eventually constant; in particular, $f$ is an element of $\overline \sd$. 
\end{proof}

Note that from the above proof we see that a sequence in $\sd$ converges to a Dehn twist in $\tau_q$ if and only if it is eventually constant.
This is not the case in $\tau_w$:
Consider the curves $a_n$ in Figure \ref{fig:converge} with the modification that the curve $a_n$ partitions off $\{0,1, n\}$ in $\bn\cup \{0\}$.
The sequence $\{T_{a_n}\} \subset \sd$ would then converge to $T_{a_1}$ in $(\pmcg(\bc\ssm(\bn\cup\{0\})),\tau_w)$, but would not converge in $\tau_q$.


\section{Inverse limits}
\label{sec:inverse}

For the entirety of this section, $S$ will be a finite-genus surface.
To simplify notation, for $\lambda \in \Lambda(S)$, we let 
\begin{align*}
G_\lambda &= \mathrm{PHomeo}^+(\overline S, \partial S, \lambda) \\
&= \{ f \in \homeo^+(\overline S, \partial S) : f(p) = p \text{ for every } p \in \lambda\}
\end{align*}
and 
\begin{align*}
G &= \mathrm{PHomeo}^+(\overline S, \partial S, \Ends(S)) \\
&= \{ f \in \homeo^+(\overline S, \partial S) : f(p)  = p \text{ for every } p \in \Ends(S)\}.
\end{align*}
Equip $\homeo(\overline S, \partial S)$ with the compact-open topology and equip $G$ and each $G_\lambda$ with the subspace topology.

Observe that $\homeo(S, \partial S)$ and $\homeo(\overline S, \partial S, \Ends(S))$ are isomorphic.
Furthermore, it is clear that the compact-open topology on the latter group must be at least as fine as the former.
However, the classical work of Arens \cite[Theorem 3]{ArensTopologies} tells us that the compact-open topology on the former group is the strongest admissible topology making it a topological group; hence, the two groups are topologically isomorphic.
In other words, we may treat the ends of $S$ as marked points.

The inclusion maps $\iota_{\lambda,\mu}\co G_\mu \hookrightarrow G_\lambda$ whenever $\lambda \subseteq \mu$ makes $(\{G_\lambda\}, \{\iota_{\lambda,\mu}\})$ an inverse system. 
We can then set $G_L$ to be the inverse limit and define $\iota_\lambda\co G_L \to G_\lambda$ to be the projections.
The initial topology on $G$ with respect to the collection $\{\iota_\lambda\}$ gives $G_L$ the structure of a topological group; this is the standard topology given to an inverse limit of topological groups.
With this topology, we see that $G_L$ is topologically isomorphic to the intersection of the $G_\lambda$, each endowed with the compact-open topology, so
\[
G_L \cong \bigcap_{\lambda \in \Lambda} G_\lambda  = G.
\]
For the remainder of the section we will identify $G$ and $G_L$.

As $\pmcg(S) = \pi_0(G)$, we want to investigate whether $\pmcg(S)$ has the structure of an inverse limit.
Similar to the above, the maps $\iota_{\lambda,\mu}$ induce forgetful maps $\vp_{\lambda, \mu} \co \pmcg(S_\mu) \to \pmcg(S_\lambda)$.

This defines an inverse system $(\{\pmcg(S_\lambda)\}, \{\vp_{\lambda,\mu}\})$ and we can take the inverse limit 
\[
L(S) = \varprojlim_{\lambda\in \Lambda} \pmcg(S_\lambda).
\]
Let $\pi_\lambda \co L(S) \to \pmcg(S_\lambda)$ be the projections defining $L(S)$.
Giving $\pmcg(S_\lambda)$ the discrete topology for each $\lambda \in \Lambda$,  we equip $L(S)$ with the initial topology with respect to $\{\pi_\lambda\}_{\lambda\in \Lambda}$.
The universal property of inverse limits gives a continuous homomorphism $\Phi\co (\pmcg(S), \tau_w) \to L(S)$.

\begin{Prop}
\label{prop:embedding}
$\Phi$ is a topological embedding.  Further, $\Phi$ is a topological isomorphism if and only if $S$ is of finite type.
\end{Prop}

\begin{proof}
Lemma \ref{lem:Hausdorff} implies that $\tau_w$ separates points and hence $\Phi$ is injective. 
It is clear from the definition that $\tau_w$ agrees with the subspace topology when $\pmcg(S)$ is viewed as a subset of $L(S)$.
Note that if $S$ is of finite type, then $L(S) = \pmcg(S)$, so for the remainder we may assume that $S$ is of infinite type.

\begin{figure}
\center
\includegraphics{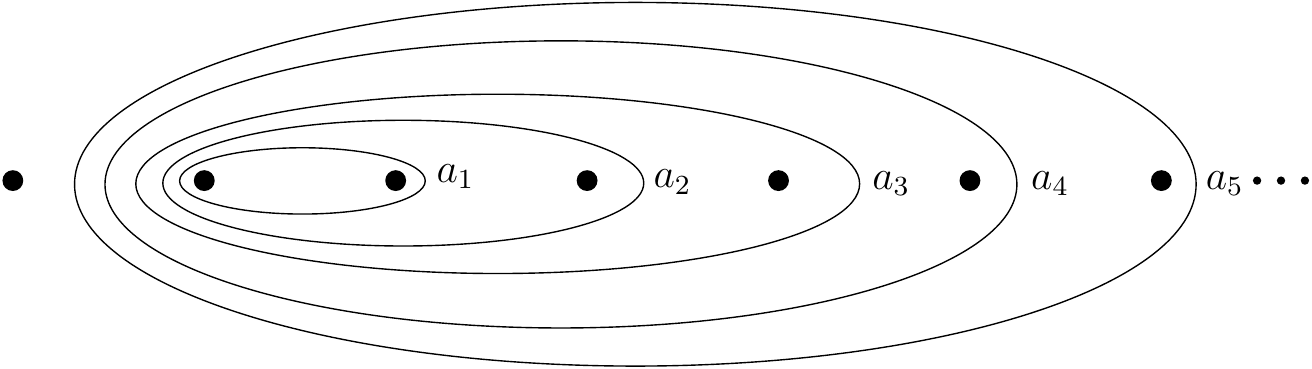}
\caption{Represented here is the surface $S=\bc\ssm\bz$.  The sequence of Dehn twists $\{T_{a_n}\}_{n\in\bn}$ converges in $L(S)$, but not in $(\pmcg(S), \tau_w)$.}
\label{fig:not-surjective}
\end{figure}

We now show that $\Phi$ fails to be onto.
We will show this in the case $S = \bc \ssm \bz$; the general case can be ascertained by replacing each integer in $\bz$ with a clopen set of $\Ends(S)$ in the general case.

For each $n \in \bn$, let $a_n$ be the boundary of a regular neighborhood of the line segment on the real axis connecting $1$ and $n+1$ as in Figure \ref{fig:not-surjective}. 
Let $T_{n}\in\pmcg(S)$ denote the Dehn twist about $a_n$.
For each $\lambda \in \Lambda$, the sequence $\{\vp_\lambda(T_n)\}$ is eventually constant, call this limit Dehn twist $T_\lambda \in \pmcg(S_\lambda)$.
It follows that the sequence $\{T_n\}_{n\in \bn}$ converges to $\overline T \in L(S)$, where $\overline T$ has coordinates $\vp_\lambda(\overline T) = T_\lambda$.

We claim $\overline T$ is not in $\pmcg(S)$; we proceed by contradiction.
Suppose $\overline T \in \pmcg(S)$ and conflate $\overline T$ with a representative homeomorphism.  
Let $b_n$ be the boundary of the ball centered at the origin with radius $n+\frac12$.
Using Lemma \ref{lem:intersection} and a similar argument as in the proof of Lemma \ref{lem:closure}, we see that $i(\overline T(b_1), b_n) > 0$.
It follows that $\overline T(b_1)$ cannot be compact, a contradiction.
\end{proof}

We point out the contrast between the pure homeomorphism group of $S$ being an inverse limit and the pure mapping class group failing to do so.
It is natural to ask if the topologies $\tau_w$ and $\tau_q$ agree as the compact-open topology and the initial topologies agree on the level of homeomorphism groups.
Proposition \ref{prop:containment} below tells us that this also fails to be the case (when $S$ is of infinite type).

Now recall that a topological group is \emph{Polish} if it is separable and completely metrizable.
Such groups have many nice properties and are studied extensively in descriptive set theory.
Of particular interest is the work of Rosendal \cite{RosendalCoarse}, which provides a theory of coarse geometry for nonlocally-compact Polish groups.

\begin{Prop}
\label{prop:polish}
$L(S)$ is Polish if and only if the cardinality of $\Ends(S)$ is countable. 
\end{Prop}

Before getting to the proof of Proposition \ref{prop:polish}, we separate out two lemmas that are interesting to us in their own right.
Observe that by the definition of the initial topology and the Dehn-Lickorish theorem, the first lemma is immediate.

\begin{Lem}
\label{lem:separable}
$L(S)$ is topologically generated by Dehn twists.
\end{Lem}

\begin{Lem}
\label{lem:first-countable}
If the cardinality of $\Ends(S)$ is uncountable, then $(\pmcg(S), \tau_w)$ is not first countable.  
\end{Lem}

The proof mimics the standard proof that an uncountable product of discrete topological spaces is not first countable.

\begin{proof}
Suppose that $\{U_n\}_{n\in \bn}$ is a countable neighborhood basis of the identity in $\pmcg(S)$.
For each $n \in \bn$ there exists a finite set $\Lambda_n  \subset \Lambda$ such that 
\[
V_n := \bigcap_{\lambda \in \Lambda_n} \ker \vp_\lambda \subset U_n.
\]
Therefore, $\{V_n\}_{n\in\bn}$ is a countable neighborhood basis for the identity.

Define $C \subset \Ends(S)$ such that $\lambda \subset C$ if and only if there exists $n\in \bn$ with $\lambda \in \Lambda_n$.
As $C$ is countable, we can find $\mu \in \Lambda$ such that $\mu \cap C = \emptyset$ and $|\mu|\geq 4$.
We claim that $\ker \vp_\mu$ is a neighborhood of the identity not containing any of  the $V_n$.
We proceed by contradiction: suppose that $V_n \subset \ker \vp_\mu$.
Pick a simple closed curve $c$ in $S$ bounding a disk $D$ in $\overline S$ disjoint from each $\lambda \in \Lambda_n$ and satisfying $|D\cap \mu| = 2$.
The Dehn twist $T_c$ about $c$ satisfies $T_c \in V_n$ and  $T_c \notin \ker\vp_\mu$.  Thus, there is no countable neighborhood base of the identity. 
\end{proof}

\begin{proof}[Proof of Proposition \ref{prop:polish}]
If $S$ is finite-type, then $L(S) = \pmcg(S)$ and has the discrete topology, so the result is immediate. 
Suppose now that $\Ends(S)$ is countably infinite.
Let $\{\lambda_n\}_{n\in \bn} \subset \Lambda$ be cofinal, so that
\[
L(S) = \varprojlim \pmcg(S_{\lambda_n}).
\]
We can then define the metric $d\co L(S) \to \br$ as follows: for any two elements $f,g \in L(S)$, set $d(f,g) = 1$ if $\pi_{\lambda_n}(f) \neq \pi_{\lambda_n}(g)$ for every $n \in \bn$, otherwise set
\[
d(f,g) = \min\{2^{-n} : \pi_{\lambda_n}(f) = \pi_{\lambda_n}(g)\}.
\]
It is a routine exercise to check that $d$ is in fact a complete metric and induces the initial topology on $L(S)$.
Lemma \ref{lem:separable} tells us that $L(S)$ is separable; hence, it is Polish.
To finish the proof, Lemma \ref{lem:first-countable} tells us that $L(S)$ is not metrizable if the cardinality of $\Ends(S)$ is uncountable. 
\end{proof}

\begin{Cor}\label{cor:baire}
If the cardinality of $\Ends(S)$ is countably infinite, then $L(S)$ is homeomorphic to the Baire space $\mathbb{N}^\mathbb{N}$.
\end{Cor}

(Another realization of the Baire space is the space of irrational numbers as a subspace of the real line.)

\begin{proof}
First observe that $L(S)$ is a subspace of a direct product of discrete spaces; hence, it has a basis of clopen sets, so it is zero dimensional.
Given that $L(S)$ is zero-dimensional and Polish, to show that $L(S)$ is homeomorphic to the Baire space it is enough to show that every compact set has empty interior; this utilizes the topological characterization of the Baire space by Alexandrov-Urysohn (see \cite[Theorem 7.7]{KechrisClassical}).
As in the proof of Proposition \ref{prop:polish}, let $\{\lambda_n\}_{n\in\bn} \subset \Lambda$ be a cofinal subset.
With this set up, the sets of the form $\pi_n^{-1}(g)$, where $\pi_n = \pi_{\lambda_n}$ and $g\in \pmcg(S_{\lambda_n})$, give a basis for the topology of $L(S)$.

Let $C\subset L(S)$ be compact.
Assume that $C$ has nonempty interior, then there exists $n \in \mathbb{N}$ and $g\in \pmcg(S_{\lambda_n})$ such that $\pi_n^{-1}(g) \subset C$.
As $\pi_n^{-1}(g)$ is closed and $C$ is compact, we must have that $\pi_n^{-1}(g)$ is compact.  
But, then $\pi_{n+1}(\pi_n^{-1}(g))$ is compact as $\pi_{n+1}$ is continuous.
This is a contradiction as $\pi_{n+1}(\pi_n^{-1}(g))$ has infinite cardinality in a discrete space.
\end{proof}

Even when $\Ends(S)$ is uncountable, $L(S)$, as an inverse limit of discrete spaces, has the structure of a complete uniform space \cite[Exercise IV.8]{IsbellUniform}. Hence, in either case, \(L(S)\) can be viewed as the completion of $\pmcg(S)$.

\subsection{Comparing topologies}
\label{sec:comparing}

Recall from Lemma \ref{lem:first-countable} that if the cardinality of $\Ends(S)$ is uncountably infinite, then $(\pmcg(S),\tau_w)$ is not metrizable, so, in this case, $\tau_w \neq \tau_q$; moreover, we have:

\begin{Prop}
\label{prop:containment}
If $S$ is of finite type, then $\tau_w = \tau_q$; otherwise, there is a strict containment $\tau_w \subset \tau_q$.
\end{Prop}

\begin{proof}
We have already noted that when $S$ is of finite type both $\tau_q$ and $\tau_w$ are discrete.
Let us assume from now on that $S$ is of infinite type.
In the above notation, for $\lambda\in \Lambda$, let $\psi_\lambda\co G \to \pmcg(S_\lambda)$ be the homomorphism factoring through $G_\lambda$.
As $\psi_\lambda$ is continuous for each $\lambda\in \Lambda$, we have that the projection homomorphism $q\co G \to (\pmcg(S), \tau_w)$ is continuous. By the definition of $\tau_q$, this shows that $\tau_w\subseteq \tau_q$.

By the remark preceding the statement of Proposition \ref{prop:containment}, we may assume that the cardinality of $\Ends(S)$ is countable. 
To finish the proof, we will find a sequence that converges in $\tau_w$ but not in $\tau_q$.
Let $K_1 \subset K_2 \subset \cdots$ be an exhaustion of $S$ by essential finite-type surfaces each of whose genus agrees with $S$.
Choose a sequence of homeomorphisms $\tilde g_n \in G$ such that $\tilde g_n$ restricts to the identity on $K_n$, which implies that $\tilde g_n \to 1$ in $G$. 
If $g_n = q(\tilde g_n)$, then $g_n \to 1$ in $\tau_q$; hence, in $\tau_w$ as well.  

\begin{figure}
\center
\includegraphics{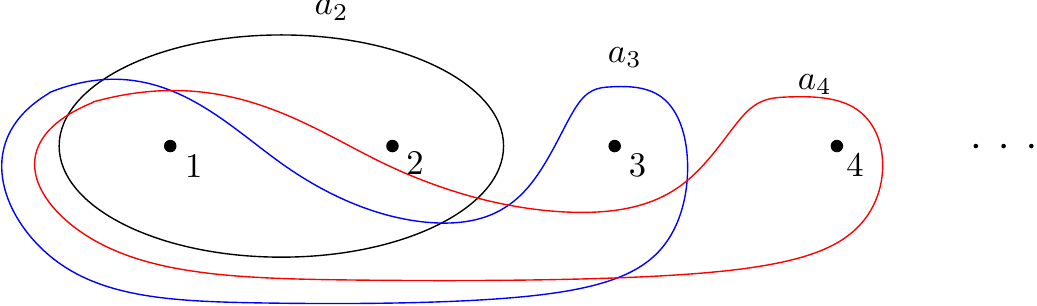}
\caption{The sequence of Dehn twists $\{T_{a_n}\}_{n=2}^\infty$ in $\pmcg(S)$, where $S = \bc\ssm\bn$, converges to the identity in $\tau_w$, but does not converge in $\tau_q$ to any homeomorphism.}
\label{fig:converge}
\end{figure}

For each $\lambda \in \Lambda$, let $n_\lambda$ be the largest integer $n$ such that $\lambda$ intersects each component of the complement of $K_{n}$ in $\overline S$.
In addition, for each $\lambda \in \Lambda$, choose an isotopy class of an essential simple closed curve in $S$, denoted by $c_\lambda$, such that $c_\lambda$ is trivial in $S_\lambda$ and $c_\lambda \cap K_{n_\lambda}$ fills $K_{n_\lambda}$.
Let $T_\lambda$ denote the Dehn twist about $c_\lambda$ and define $h_\lambda \in \pmcg(S)$ to be $T_\lambda\circ g_{n_\lambda}$. 

The map $\Lambda \to \pmcg(S)$ defined by $\lambda \mapsto h_\lambda$ gives a net in $\pmcg(S)$.
Now, the collection of sets of the form $U_\mu  = \ker \vp_\mu$ for $\mu \in \Lambda$ form a neighborhood subbasis for the identity in $\tau_w$.
If $\lambda \supset \mu$, then we claim $h_\lambda \in U_\mu$. 
First, observe that $K_{n_\mu}\subset K_{n_\lambda}$, so that $\tilde g_{n_\lambda}$ restricted to $K_{n_\mu}$ is the identity.
It follows that $\vp_\mu (g_{n_\lambda})$ is trivial.
Also,
\[
\vp_\mu(T_\lambda) = \vp_{\mu, \lambda}\circ \vp_\lambda(T_\lambda)
\]
is the identity as $T_\lambda \in \ker \vp_\lambda$ by definition. 
This shows that $h_\lambda \in U_\mu$ for all $\lambda \supset \mu$.
In particular, the net $\{h_\lambda\}_{\lambda \in \Lambda}$ converges to the identity in $\tau_w$.  

We claim that this net does not converge in $\tau_q$.
Let 
\[
U = \{ f \in \pmcg(S) : f = q(\tilde f) \text{ where } \tilde f(K_1) = K_1 \},
\]
then $U$ is an open neighborhood of the identity in $\tau_q$.
Indeed, let $V$ be an open regular neighborhood of $K_1$, then $U$ is the image, under $q$, of 
\[
U(K_1, V) = \{\tilde f \in G : \tilde f (K_1) \subset V\},
\]
which is a basis element for the compact-open topology on $G$.
Further, as $q$ is both a homomorphism and a quotient map, it is open; hence $U$ is open.
Now observe that $h_\lambda \notin U$ for any $\lambda$; this follows from the fact that given any simple closed curve $c \subset K_1$, we have $h_\lambda(c) = T_\lambda(c)$, which intersects $\partial K_1$ nontrivially.    
\end{proof}

\bibliographystyle{amsplain}
\bibliography{references}

\end{document}